\documentclass[12pt,oneside]{amsart} 
\usepackage{amssymb,amscd}
\usepackage{euscript}

      \theoremstyle{plain}
      \newtheorem{theorem}{Theorem}[section]
      \newtheorem{lemma}[theorem]{Lemma}
      \newtheorem{corollary}[theorem]{Corollary}
      \newtheorem{proposition}[theorem]{Proposition}

      \newtheorem{definition}[theorem]{Definition}

\numberwithin{equation}{section}

      \makeatletter
      \def\@setcopyright{}
      \def\serieslogo@{}
      \makeatother

\def\A{\EuScript{A}} 
\def\B{\EuScript{B}} 
\def\F{\EuScript{F}} 
\def\E{\mathcal{E}}
\def\M{\mathcal{M}}
\def\R{\mathbb R}
   
\def\C{\mathbb C}
\def\Z{\mathbb Z}
\def\N{\mathbb N}
\def\Q{\mathbb Q}
\def\T{\mathbb T}
\def\bv{\mathbf v}
\def\tC{\tilde C}
\def\tB{\tilde \B}
\def\dist{\text{dist}}

\def\Id{\text{Id}}
\def\e{\epsilon}

\def\a{\alpha}

\def\QED{\hfill\hfill{\square}}

\begin{document}

\author{Victoria Sadovskaya$^{*}$}

\address{Department of Mathematics, The Pennsylvania State Ubiversity, University Park, PA 16802, USA.}
\email{sadovskaya@psu.edu}

\title [Cohomology  of fiber bunched cocycles over hyperbolic systems]
{Cohomology  of fiber bunched cocycles \\ over hyperbolic systems} 

\thanks{$^{*}$ Supported in part by NSF grant DMS-1301693}

\begin{abstract} 
We consider  H\"older continuous fiber bunched $GL(d,\R)$-valued cocycles 
over an Anosov diffeomorphism.  We show that two such cocycles are 
H\"older continuously cohomologous if they have equal periodic data, 
and prove a result for cocycles with conjugate periodic data. 
We obtain a corollary for cohomology between {\em any} constant cocycle and its small perturbation.
The fiber bunching condition  means that non-conformality of the cocycle is dominated by the expansion and contraction in the base. We show that this condition can be established based on the periodic data. Some important examples of cocycles
come from the differential of the diffeomorphism and its  restrictions 
to invariant sub-bundles. We discuss an application of our results to the question when an Anosov diffeomorphism is smoothly
 conjugate to a $C^1$-small perturbation.  We also establish H\"older continuity of a measurable conjugacy between a fiber bunched cocycle and a 
 uniformly quasiconformal one. Our main results also hold for cocycles 
with values in a closed subgroup of $GL(d,\R)$,  for cocycles 
over hyperbolic sets and shifts of finite type, and for linear cocycles 
on a non-trivial vector bundle.
  
 \end{abstract}

\maketitle


\section{Inroduction}
Cocycles and their cohomology arise naturally in the theory of group 
actions and play an important role in dynamics. In this paper we study 
cohomology of H\"older continuous group-valued cocycles over hyperbolic 
dynamical systems.
Our motivation comes in part from questions in local and global rigidity 
for hyperbolic systems and actions, where the derivative and  the Jacobian provide important examples of cocycles.
We state our results for the case of an Anosov diffeomorphism, but
they also hold for cocycles over hyperbolic sets and symbolic dynamical 
systems.

\begin{definition} Let $f$ be a diffeomorphism of a compact 
manifold $\M$ and let $A$  be a H\"older continuous function 
from $\M$ to a  metric group  $G$. The $G$-valued cocycle 
over $f$ generated by $A$ is the map 
 $\A:\M \times \Z \to G$ defined  by 
 $$
\begin{aligned}
&\A(x,0)=\A_x^0=e_G,  \;\;\;\; \A(x,n)=\A_x^n = 
A(f^{n-1} x)\circ \cdots \circ A(x) \quad \text{ and }\;\; \\
&\A(x,-n)=\A_x^{-n}= (\A_{f^{-n} x}^n)^{-1} = 
 (A(f^{-n} x))^{-1}\circ \cdots \circ (A(f^{-1}x))^{-1}, \quad n\in \N.
\end{aligned}
$$
\end{definition}


\noindent If the tangent bundle of $\M$ is trivial, i.e. $T\M= \M\times \R^d$, 
then the differential $Df$ can be viewed as a $GL(d,\R)$-valued cocycle:
$  A_x=Df_x$ and  $\A_x^n=Df^n_x.$
 More generally, one can consider restrictions of $Df$
to invariant sub-bundles of $T\M$, for example stable and unstable. Typically, these sub-bundles are only H\"older continuous, and hence 
so are the corresponding cocycles. On the other hand, H\"older regularity 
is necessary to develop a meaningful theory for cocycles over 
hyperbolic systems, 
even in the simplest case of $G=\R$.

\begin{definition} \label{cohdef}
Cocycles $\A$ and $\B$ are 
(measurably, continuously)  {\em cohomologous} 
if there exists a (measurable, continuous) 
function  $C:\M\to G$ such that
\begin{equation}\label{conj eq}
  \A_x^n=C(f^n x) \circ \B_x^n \circ  C(x)^{-1} 
  \quad\text{ for all }n\in \Z \text{ and }x\in \M,
\end{equation}
equivalently, $\A_x=C(fx) \circ \B_x \circ C(x)^{-1}\,\text{ for all }x\in \M.$ 

We refer to $C$ as a {\em conjugacy} between $\A$  and $\B$.
It is also called a transfer map.
\end{definition}

H\"older continuous cocycles over hyperbolic systems have been extensively  studied starting with the seminal work of A. Liv\v{s}ic \cite{Liv1,Liv2}. 
The research has been focused on obtaining sufficient conditions
for cohomology in terms of the periodic data and on studying the regularity
of the conjugacy $C$, see \cite{KtN} for an overview.


\begin{definition}  \label{def cong}  
Cocycles $\A$ and $\B$ have\, {\em conjugate periodic data} if
for every periodic point $p=f^n(p)$ in $\M$
 there exists  $C(p)\in G$ such that 
\begin{equation}\label{conj periodic}
\A_p^n=C(p) \circ \B_p^n \circ C(p)^{-1}.
\end{equation}
\end{definition}

Clearly, having conjugate periodic data is a necessary condition 
for continuous cohomology of two cocycles, and it is natural 
to ask whether it is also sufficient.
If $G$ is an abelian group, the problem reduces to the case 
when $\B$ is the identity cocycle,  i.e. $\B_x=e_G$, and the periodic assumption  is simply $ \A_p^n=e_G$. The positive answer 
for this case was given by A. Liv\v{s}ic \cite{Liv1}.
Even for non-abelian $G$, the case of $\B=e_G$ has been 
studied most and by now is relatively well understood, see for example 
\cite{Liv2,NT95,PW,LW,K11}.

For non-abelian $G$, however, the general problem does not reduce to 
the special case $\B=e_G$ and is much more difficult. 
There are very few results for non-abelian groups, 
and almost none beyond the essentially compact case. 
Even when $C(p)$ is bounded the answer 
is negative in general \cite{S12}. If $C(p)$ is  H\"older, conjugating 
$\B$ by the extension of $C$ reduces the problem to the case
of equal periodic data, i.e. $\A_p^n= \B_p^n$. Positive results for  
equal periodic data, as well as some results for conjugate data, 
were established by W.~Parry \cite{Pa} for compact $G$ and, 
somewhat more generally, by K. Schmidt \cite{Sch} for cocycles 
with ``bounded distortion". First results outside this setting 
were obtained in \cite{S12} for certain types of $GL(2,\R)$-valued cocycles.

\vskip.1cm

In this paper we consider fiber bunched  cocycles with values in $GL(d,\R)$
or its closed subgroup.
We establish H\"older cohomology for cocycles with equal periodic data
and prove a result for cocycles with conjugate periodic data 
under a mild regularity assumption on $C(p)$.
The fiber bunching condition  \eqref{fiber bunched} means that non-conformality of the cocycle is, in a sense, dominated by expansion and contraction in the base.  In particular, conformal and uniformly 
quasiconformal cocycles satisfy this condition.
Fiber bunching and similar assumptions ensure convergence of certain 
iterates of the cocycle and play a crucial role in the non-commutative case.  We show that fiber bunching can be obtained from the 
periodic data, and hence we assume it for only one of the cocycles.
We  obtain a corollary for perturbations of {\em any} constant 
cocycle, not necessarily fiber bunched.

We also consider a related question whether a measurable solution $C$
of \eqref{conj eq} is necessarily continuous. 
Even the case of $\B=e_G$ remains open in full generality, but  positive answers were obtained under additional assumptions
\cite{Liv2, GS, NP, PW}. 
The case of two arbitrary cocycles with values in a compact 
group was resolved affirmatively by  W.~Parry and M. Pollicott  \cite{PP}, 
and by K. Schmidt \cite{Sch} for cocycles 
with ``bounded distortion".
Positive results for certain types of $GL(2,\R)$-valued cocycles 
 were obtained in \cite{S12}. On the other hand,
examples of $GL(2,\R)$-valued cocycles which are measurably but not continuously cohomologous were constructed in \cite{PW}, moreover both cocycles can be made arbitrarily close to the identity.
This shows that fiber bunching of the cocycles 
does not ensure continuity of $C$. In this paper we establish H\"older
 continuity of a measurable conjugacy under a stronger assumption 
 that one cocycle is fiber bunched and the other one is uniformly
quasiconformal. For smooth cocycles, higher regularity of the conjugacy
then follows from \cite{NT98}.


We state the results on cohomology of cocycles in Section 2
and and give the proofs in Section~\ref{proofs}.
We describe other settings for our results in Section \ref{settings}.
In Section \ref{localrig} we discuss an application to 
the question when an Anosov diffeomorphism 
is smoothly conjugate to a $C^1$-small perturbation.

We would like to thank Boris Kalinin for helpful discussions.


\section{Statement of results on cohomology of cocycles} 

\noindent{\bf Anosov diffeomorphisms.} Let $\M$ be a compact connected Riemannian manifold.
We recall that a diffeomorphism  $f$ of  $\M$
 is called {\it Anosov}\, if there exist a splitting 
of the tangent bundle $T\M$ into a direct sum of two $Df$-invariant 
continuous subbundles $E^s$ and $E^u$,  a Riemannian 
metric on $\M$, and  continuous  
functions $\nu$ and $\hat\nu$  such that 
\begin{equation}\label{Anosov def}
\|Df(\bv^s)\| < \nu(x) < 1 < \hat\nu(x) <\|Df(\bv^u)\|
\end{equation}
for any $x \in \M$ and unit vectors  
$\,\bv^s\in E^s(x)$ and $\,\bv^u\in E^u(x)$.
The distributions $E^s$ and $E^u$ are called stable and unstable. 
They are tangent to the stable and unstable foliations 
$W^s$ and $W^u$ respectively (see, for example \cite{KH}).
A diffeomorphism is said to be {\it transitive} if there is a point $x$ in $\M$
with dense orbit. All known examples of Anosov diffeomorphisms have this property.
\vskip.3cm

\noindent {\bf Standing assumptions.} In this paper,

{\it \noindent $f$ is a $C^2$ transitive Anosov diffeomorphism 
of a compact connected manifold $\M$,

\noindent  $\A$ and $\B$ are $\beta$-H\"older continuous 
$GL(d,\R)$-valued cocycles over  $f$.}
\vskip.2cm

 We denote by  $\| A\|$  the operator norm of the matrix 
$A$ and 
we use the following distance on  $GL(d,\R)$:
$
\; d(A, B) = \|A-B\| + \|A^{-1}-B^{-1}\|.
$
\vskip.1cm
A $GL(d,\R)$-valued cocycle $\A$ is $\beta$-H\"older continuous
if there exist constant $c$ such that 
$\,d(\A_x, \A_y) \le c\, \dist(x,y)^\beta \,$ for all $x,y\in \M$.

\begin{definition} A $\beta$-H\"older continuous cocycle $\A$ 
over an Anosov diffeomorphism $f$ is\,
 {\em fiber bunched} if 
there exist numbers $\theta<1$ and $L$  such that for all $x\in\M$ and $n\in \N$,
\begin{equation}\label{fiber bunched}
\| \A_x^n\|\cdot \|(\A_x^n)^{-1}\| \cdot  (\nu^n_x)^\beta < L\, \theta^n \;\text{ and}\quad
\| \A_x^{-n}\|\cdot \|(\A_x^{-n})^{-1}\| \cdot  (\hat \nu^{-n}_x)^\beta < L\, \theta^n,
\end{equation}
where 
$ \;\nu^n_x=\nu(f^{n-1}x)\cdots\nu(x) \,\text{ and }\;
   \hat\nu^{-n}_x=(\hat \nu(f^{-n}x))^{-1}\cdots  (\hat\nu(f^{-1}x))^{-1}.$
\end{definition}

First we establish H\"older cohomology for cocycles with equal periodic data.

\begin{theorem} \label{equal data}
Suppose that  a cocycle  $\A$ is fiber bunched 
and  a cocycle $\B$ has the same periodic data, 
i.e. $\B_p^n=\A_p^n$ whenever $f^n(p)=p$. 
Then $\A$ and $\B$ are
$\beta$-H\"older continuously cohomologous.
Moreover, if $\A$ and $\B$ take values in a closed subgroup 
of $GL(d,\R)$, then a  $\beta$-H\"older continuous
conjugacy between them can be chosen in the same subgroup.
\end{theorem}
    
In this theorem we assume fiber bunching only for $\A$, as for $\B$ 
it follows from the proposition below. We give a necessary 
and sufficient condition for a cocycle
to be  fiber bunched in terms of its periodic data in 
Corollary  \ref{periodic bunching}.

\begin{proposition} \label{fiber bunching from periodic} 
Suppose that a cocycle $\A$ is  fiber bunched and $\B$ has
conjugate periodic data.
Then $\B$ is also fiber bunched.
\end{proposition}

Now we consider the question whether conjugacy 
of the periodic data for two cocycles implies cohomology.
The case of H\"older congugacy of the periodic data 
easily reduces to the case of equality. Indeed, 
one can extend the H\"older continuous function $C(p)$ 
to $\M$ and consider the cocycle 
$\tilde \B_x= C(fx)\circ \B_x \circ C(x)$
so that $\A$ and $\tilde \B$ have equal periodic data.
By  Theorem \ref{equal data}
the cocycles $\A$ and $\tilde \B$ are H\"older cohomologous,
and hence so are $\A$ and $ \B$.

On the other hand, Example 2.7 in \cite{S12}  shows that  
boundedness assumption  for the conjugacy is too weak:
arbitrarily close to the identity, there exist smooth
$GL(2, \R)$-valued cocycles 
 that have conjugate periodic data
 with $C(p)$ uniformly bounded, 
but are not even measurably cohomologous. 

In the next theorem we assume that the diffeomorphism $f$
has a fixed point. It is an open question whether every Anosov diffeomorphism satisfies this assumption.
We obtain H\"older cohomology of the cocycles 
if $C(p)$ is H\"older continuous at a fixed point. 
If we assume that $C(p)$ is H\"older continuous
at a periodic point $p=f^Np$, then the theorem yields 
H\"older cohomology of the iterates $\A^N$ and $\B^N$ over $f^N.$

\begin{theorem} \label{conjugate  data}
Suppose that $\A$ is fiber bunched and $\B$ has conjugate periodic data.
In addition, suppose that $f$ has a fixed point $p_0$  and the conjugacy $C(p)$ 
is $\beta$-H\"older continuous at  $p_0$,
i.e.  $d(C(p), C(p_0)) 
 \le c\, \dist (p,p_0)^\beta$ for every periodic point $p$.

Then $C(p_0)$ extends to a unique $\beta$-H\"older continuous conjugacy 
$C$ between $\A$ and $\B$.
Moreover, if $\A$, $\B$, and $C(p_0)$ take values in a closed subgroup  $G_0$ of $GL(d,\R)$, then $C(x)\in G_0$ for all $x$. 
\end{theorem}

The corollary below gives a similar result for a constant cocycle
and its perturbation without the fiber bunching assumption.
The proof is outlined in the end of Section \ref{localrig}.

\begin{corollary} \label{constant}
Suppose that $\A$ is a constant cocycle,  and  $\B$
is sufficiently close to $\A$ and has conjugate periodic data.
In addition, suppose that $f$ has a fixed point $p_0$  
and  $C(p)$ 
is H\"older continuous at  $p_0$.
Then $\A$ and $\B$ are 
H\"older continuously cohomologous.
\end{corollary}

Next we consider the question whether a measurable conjugacy 
between two fiber bunched cocycles is continuous.
An example in \cite{PW} demonstrates that the answer is negative in general: arbitrarily close to the identity, there exist smooth 
$GL(d,\R)$-valued cocycles that are are measurably,
but not continuously cohomologous.
Thus we make a stronger assumption that one of 
the cocycles is uniformly quasiconformal.

\begin{definition} \label{quasiconformal}
A cocycle $\B$ is called {\em uniformly quasiconformal}
 if the quasiconformal distortion 
$\,K_\B(x,n)=\|\B^n_x\| \cdot \|(\B^n_x)^{-1}\|\,$
 is uniformly bounded for all $x\in \M$ and $n\in \Z$.
 If $\,K_\B(x,n)=1$ for all $x$ and $n$,
 the cocycle is said to be \em{conformal}.
 
\end{definition}

\begin{theorem} \label{measurable conjugacy}
Suppose that $\A$ is fiber bunched and 
$\B$ is uniformly quasiconformal. Let  $\mu$ be an ergodic
invariant  measure with full support and local product structure.

Then any $\mu$-measurable conjugacy between $\A$ and $\B$ is $\beta$-H\"older continuous,
i.e. it coincides with a $\beta$-H\"older continuous conjugacy on a set of full measure.

\end{theorem}

A measure has local product structure if it is locally equivalent to the product of its conditional measures on the local stable and unstable manifolds.
Examples of ergodic measures with full support and local product structure include 
the measure of maximal entropy,  more generally Gibbs (equilibrium) measures of H\"older continuous potentials, and the invariant volume  if it exists \cite{PW}. 

\vskip.3cm


\section{Other settings} \label{settings}

\noindent{\bf Other systems in the base.} Our results hold and  the proofs apply
without significant modifications to $GL(d,\R)$ -valued cocycles over mixing locally 
maximal hyperbolic sets and over mixing shifts of finite type. 
Mixing holds automatically for transitive Anosov diffeomorphisms 
of connected manifolds. We briefly describe the other two settings. 
\vskip.2cm

\noindent{\bf 1. Cocycles over hyperbolic sets.} 
(See \cite{KH} for more details.)
Let $f$ be a diffeomorphism of a manifold $\M$.
A compact $f$-invariant  set $\Lambda \subset \M$ is
called {\em hyperbolic} if there  exist a continuous $Df$-invariant splitting 
$T_\Lambda \M = E^s\oplus E^u$, and a Riemannian metric and continuous functions $\nu$, $\hat \nu$ on an open set
$U \supset \Lambda$ such that 
\eqref{Anosov def} holds for all $x \in \Lambda$.
A $\beta$-H\"older  cocycle over the map $f|_{\Lambda}$
is fiber bunched if \eqref{fiber bunched} holds on $\Lambda$.

The set $\Lambda$ is called {\em locally maximal} if 
$\Lambda= \bigcap_{n\in \Z} f^{-n }(U)$ for some open set $U\supset \Lambda$. The map $f|_\Lambda$ is called {\em topologically mixing}\,
 if for any two open non-empty subsets $U,V$ of $\Lambda$
 there is $N\in \N$ such that $\, f^n(U)\cap V\ne \emptyset\,$ for all $n\ge N$.

\vskip.3cm
\noindent{\bf 2. Cocycles over shifts of finite type.}
Let $Q$ be $k \times k$ matrix with entries from $\{ 0,1 \} $ such that all 
entries of $Q^N$ are positive for some $N$. Let
$$
\Sigma = \{ \,x=(x_n) _{n\in \Z}\;\;  | \;\; 1\le x_n\le k \;\text{ and }\;
 Q_{x_n,x_{n+1}}=1 \,\text{ for every } n\in \Z \,\}.
$$
\noindent The shift map $\sigma:\Sigma \to \Sigma\,$ is defined by 
$(\sigma(x))_n=x_{n+1}$.
The system $(\Sigma,\sigma)$ is called a mixing {\em  shift of finite type}. 
$\Sigma$ has a natural family of 
metrics $d_\alpha$, $\alpha \in (0,1)$,
defined by 
$$
d_\alpha(x,y)=\alpha^{n(x,y)},
\;\text{ where }\;n(x,y)=\min\,\{ \,|i|\;\;| \;\, x_i \ne y_i  \}.
$$
The following sets play the role of the local stable and unstable 
manifolds of $x$:
$$
W^s_{loc}(x)=\{\,y\; | \;\, x_i=y_i, \;\;i\ge 0\,\}, \quad 
W^u_{loc}(x)=\{\,y\; | \;\, x_i=y_i, \;\;i\le 0\,\},
$$
indeed for $n\in \N$,
$$d_\alpha (\sigma^n(x), \sigma^n(y) )= \a^n \, d_\alpha (x,y) 
\quad\text{for }y\in W^{s}_{loc}(x),$$
$$d_\alpha (\sigma^{-n}(x), \sigma^{-n}(y) )= \a^n\, d_\alpha (x,y)
\quad\text{for }y\in W^{u}_{loc}(x).
$$
Hence the main distance estimate \eqref{dist}  in our proofs 
holds with $\nu=\a$ and $\hat \nu =1/\alpha$.
 A $\beta$-H\"older cocycle $\A$ over $(\Sigma,\sigma, d_\a)$ is  fiber bunched if there are $\theta<1$ and $L$ such that
$$
\| \A_x^n\|\cdot \|(\A_x^n)^{-1}\| \cdot  \alpha^{\beta |n|} < L\, \theta^{|n|}
\quad \text {for all }n\in \Z.
$$
\vskip.2cm

\noindent{\bf Linear cocycles over an Anosov diffeomorphism.} 
A $GL(d,\R)$-valued cocycle over $f$
can be viewed as an automorphism of the trivial vector bundle 
$\E =\M \times \R^d$. More generally, we can consider
{\em linear cocycles} over $f$, i.e. automorphisms of a $d$-dimensional vector bundle $\E$ over $\M$ covering $f$, see \cite{KS12} for details of this setting including H\"older regularity. 
The results (except for statements about subgroups) and the proofs extend directly to this context.


\section{proofs} \label{proofs}


\subsection{Fiber bunching and periodic data}  $\;$
In this section we prove Proposition \ref{fiber bunching from periodic}
and then we formulate the fiber bunching condition in terms of 
the periodic data.

\vskip.1cm

\noindent{\bf Proof of Proposition \ref{fiber bunching from periodic}}.
The proof relies on the following result on subadditive sequences.
Let $f$ be a homeomorphism of a compact metric space $X$.
A sequence of continuous functions $a_n : X \to \R$ is called
{\em subadditive}\, if 
$$
  a_{n+k}(x) \le a_k(x) +  a_n(f^k x) 
\quad\text{for all }x\in X \text{ and  }n,k\in \N.
$$
Let $\mu$ be an $f$-invariant Borel  probability measure on $X$ and let
$a_n(\mu)= \int_X a_n d\mu$. Then $a_{n+k}(\mu) \le a_n(\mu) + a_k(\mu)$, 
i.e. the sequence 
of real numbers $\{ a_n (\mu) \}$ is subadditive. It is well known that
for such a sequence the following limit exists:
$$
 \chi(a,\mu) :=\lim_{n\to \infty} \frac {a_n (\mu)}n 
 = \inf_{n\in \N} \frac {a_n (\mu)}n.
$$
Also, by the Subaddititive Ergodic Theorem, if the measure $\mu$ is ergodic 
then 
$$
 \lim_{n\to \infty} \frac {a_n (x)}n= \chi(a,\mu) \quad
  \text{for $\mu$-almost all } x\in X.
$$ 

\begin{lemma}\cite[Proposition 4.9]{KS12} \label{subadditive}
Let $f$ be a homeomorphism of a compact metric space $X$
and  $a_n : X \to \R$  be a subadditive sequence of continuous functions.

If $\chi(a,\mu)<0$  for every ergodic invariant Borel probability measure 
$\mu$ for $f$, then there exists $N$ such that $a_N(x) <0$ for all $x \in X$.
\end{lemma}

We will apply this result to the  sequence of functions
$$
  a_n(x)= \log \,(\| \B_x^n \| \cdot \| (\B_x^n)^{-1}\| \cdot (\nu^n_x)^\beta ).
$$
It is easy to verify that this sequence is subadditive.
To show that it satisfies the assumption of  Proposition \ref{subadditive},
we consider Lyapunov exponents of cocycles.

Let $\mu$ be an ergodic $f$-invariant measure, and let 
 $\lambda_+(\B, \mu)$ and $\lambda_-(\B, \mu)$ be
the largest and smallest Lyapunov exponents of $\B$ with 
respect to $\mu$. We recall that 
$$
  \lambda_+(\B, \mu)= \,\lim_{n\to \infty} \frac1n \log \| \B^n_x \| \quad\text{and}\quad
    \lambda_-(\B, \mu)= \,\lim_{n\to \infty} \frac1n \log \| (\B^n_x)^{-1} \|^{-1}
$$
for $\mu$ almost every $x\in \M$ (see \cite[Section 2.3]{BP},  for more details).

Let $p=f^k p$ be a periodic point for $f$. The 
largest and smallest Lyapunov exponents of $\B$ 
with respect to the invariant measure $\mu_p$ on the orbit of $p$ satisfy
$$
 \lambda_{\pm}(\B, \mu_p)= 
 \,\frac1k \log \,\left(\text{the largest/smallest $|$eigenvalue of }\B_p^k\,| \right).
$$
Since the matrices $\A^k_p$ and $\B^k_p$ are conjugate,
it follows that  $\lambda_{\pm}(\B, \mu_p)=\lambda_{\pm}(\A, \mu_p)$.

For the scalar cocycle $\nu^\beta$, 
$\, \lambda(\nu^\beta, \mu)=\int_\M \log \nu(x)^\beta\,d\mu$
by the Birkhoff Ergodic Theorem, in particular 
$\lambda(\nu^\beta, \mu_p)=\frac1{k} \log (\nu^k_p)^\beta$.
\vskip.1cm

Since the cocycle $\A$ is fiber bunched, there are numbers $L$ and $\theta<1$ 
such that
$$
\| \A_x^n \| \cdot \| (\A_x^n)^{-1}\| \cdot (\nu^n_x)^\beta < L\,\theta^n
$$
for every $x\in \M$ and $n\in \N$. It follows that 
$$
  \lambda_+(\A, \mu_p)-  \lambda_-(\A, \mu_p) +\lambda(\nu^\beta, \mu_p)  =
  \lim_{n\to \infty} \frac1n \log (\| \A_p^n \| \cdot \| (\A_p^n)^{-1}\| \cdot (\nu^n_p)^\beta )
  \le\log \theta <0,
$$
and hence 
$$
  \lambda_+(\B, \mu_p)-  \lambda_-(\B, \mu_p) +\lambda(\nu^\beta, \mu_p)  
  \le\log \theta <0.
$$

We consider the cocycle $\F=\B \oplus \nu$ over $f$.
By \cite[Theorem 1.4]{K11}, the Lyapunov exponents
 $\lambda_1 \le ... \le \lambda_d$ of $\F$ with respect 
 to an ergodic invariant measure $\mu$  (listed with multiplicities) 
 can be approximated 
by the Lyapunov exponents of $\F$ at periodic points. More precisely,
for any $\e >0$ there exists a periodic point $p \in \M$ for which the 
Lyapunov exponents $\lambda_1^{(p)} \le ... \le \lambda_d^{(p)}$ of $\F$ 
satisfy $|\lambda_i-\lambda_i^{(p)}|<\e$ for $i=1, \dots , d$.
\vskip.1cm

Thus for the sequence of functions 
$a_n(x)=\log \,(\| \B_x^n \| \cdot \| (\B_x^n)^{-1}\| \cdot (\nu^n_x)^\beta )$, 
$$
 \chi(a,\mu) \overset{def}{=}\lim_{n\to\infty} \frac{a_n(x)}{n}= 
 \lambda_+(\B, \mu)-  \lambda_-(\B, \mu) +\lambda(\nu^\beta, \mu)  < 0.
 $$
Now it follows from Lemma \ref{subadditive} that there exists 
$N$ such that $a_N(x)<0$ for all $x$, i.e.
\begin{equation}\label{B^N}
\| \B_x^N \| \cdot \| (\B_x^N)^{-1}\|\cdot (\nu^N_x)^\beta <1
\quad\text{for all }x\in \M.
\end{equation}
By continuity, there exists $\tilde \theta<1$
such that the left hand side of \eqref{B^N} is smaller than $\tilde \theta$ for
all $x$.
Writing  $n\in \N$  as $n=mN+r$, where $m\in \N\cup \{0\}$ and $0\le r<N$, 
we get 
\begin{equation}\label{eq theta}
  \| \B_x^n \| \cdot \| (\B_x^n)^{-1}\|\cdot (\nu^n_x)^\beta \le L\, \tilde \theta^{\,m},
  \,\text { where }\, L=\max_{x,r}\, (\| \B_x^r \| \cdot \| (\B_x^r)^{-1}\|\cdot (\nu^r_x)^\beta).
\end{equation}

The corresponding inequality  with $\hat\nu$  
is obtained similarly, and 
we conclude that the cocycle $\B$ is fiber bunched.
$\QED$

\vskip.2cm

The argument implies the following.

\begin{corollary} \label{periodic bunching} A cocycle $\B$ is fiber bunched if and only if 
there exists a number $\eta<0$ such that 
for every $f$-periodic point $p=f^k p$,
$$
\lambda_+(\B, \mu_p)-  \lambda_-(\B, \mu_p) +\lambda(\nu^\beta, \mu_p) =
  \,\frac1k \log \left(\frac{\text {largest $|$eigenvalue of }\B_p^k\,| }
  {\text{smallest $|$eigenvalue of }\B_p^k\,| } 
  \, (\nu_p^k)^\beta \right)<\eta
$$
and the corresponding enequality holds for $\hat \nu$.
\end{corollary}


\vskip.3cm

\subsection{Holonomies}
An important role in our arguments is played by holonomies.  We follow the notations and terminology form \cite{V,ASV} for 
linear cocycles.
\vskip.1cm

Let $\E=\M\times \R^d$ be a trivial vector bundle over $\M$. 
We view $\A_x$ as a linear map from $\E_x$, the fiber at $x$, 
to $\E_{fx}$,\, so  $\,\A_x^n: \E_x \to \E_{f^nx}\,$ and 
$\,\A_x^{-n}: \E_x \to \E_{f^{-n}x}$.

\begin{definition} \label{holonomies}
A {\em stable holonomy} for a linear cocycle $\A:\E \to \E$ is 
a continuous map $H^{\A,s}:\;(x,y)\mapsto H^{\A,s}_{x,\, y}$,\, 
where $x\in \M$, $y\in W^s(x)$, such that  
\begin{itemize}
\item[(H1)] $H^{\A,s}_{x,\,y}$ is a linear map from $\E_x$ to $\E_y$;
\vskip.1cm
\item[(H2)]  $H^{\A,s}_{x,\,x}=\Id\,$ and $\,H^{\A,s}_{y,\,z} \circ H^{\A,s}_{x,\,y}=H^{\A,s}_{x,\,z}$;
\vskip.1cm
\item[(H3)]  $H^{\A,s}_{x,\,y}= (\A^n_y)^{-1}\circ H^{\A,s}_{f^nx ,\,f^ny} \circ \A^n_x\;$ 
for all $n\in \N$.
\end{itemize}
\end{definition}

\noindent Condition (H2) implies that $(H^{\A,s}_{x,\,y})^{-1}=H^{\A,s}_{y,\,x}.$

\noindent The unstable holonomy $H^{\A,u}$ are defined 
similarly for $y\in W^u(x)$ with 
\vskip.2cm

(H3$'$)  $H^{\A,u}_{x,\,y}=
(\A^{-n}_y)^{-1} \circ H^{\A,u}_{f^{-n}x ,\,f^{-n}y}  \circ \A^{-n}_x\;$ 
 for all $n\in \N$.
\vskip.2cm

\noindent We consider holonomies which  satisfy the following H\"older condition:
\vskip.2cm 
(H4) $\;\;\| H^{\A,s(u)}_{\,x,y} - \Id \,\| \leq c\,\dist (x,y)^{\beta},
\; \text{ where }c  \text{ is independent of } x  \text{ and } y\in W^{s(u)}_{loc}(x).$
\vskip.2cm 

\noindent A local stable manifold $W^{s}_{loc}(x)$ is a ball in $W^{s}(x)$ centered 
at $x$  of a small radius $\rho$ in the intrinsic metric of $W^{s}(x)$.
We choose  $\rho$ small enough
 so that
 \eqref{Anosov def} ensures that $\| Df_y\| < \nu (x)$
for all $x \in \M$ and $y \in W^{s}_{loc}(x)$. Local unstable manifolds 
are defined similarly, and it follows that for all $n\in \N$,
\begin{equation}\label{dist}
\begin{aligned}
&\dist (f^nx, f^ny)< \nu^n_x \cdot  \dist(x,y) \quad \text{for all } x \in \M
\text{ and }y\in W^s_{loc}(x),\\
&\dist (f^{-n}x, f^{-n}y)< \hat \nu^{-n}_x \cdot  \dist(x,y) \quad \text{for all } x \in \M
\text{ and }y\in W^u_{loc}(x).
\end{aligned}
\end{equation}

\begin{proposition} \label{existence of holonomies} 
Suppose that a cocycle $\A$ is fiber bunched. 
Then $\A$ has unique stable  and unstable holonomies satisfying 
(H4). Moreover, for every $x\in \M$,
$$
\begin{aligned}
& H^{\A,s}_{x,y} =\underset{n\to\infty}{\lim} (\A^n_y)^{-1} \circ \A^n_x, 
 \quad  y\in W^s(x), \quad \text{and}\\
&H^{\A,u}_{x,\,y}\,=
\underset{n\to\infty}{\lim} \left( (\A^{-n}_y)^{-1} \circ (\A^{-n}_x) \right)
\,=
\underset{n\to\infty}{\lim} 
\left(\A^n_{f^{-n}y} \circ (\A^{n}_{f^{-n}x})^{-1} \right), \quad y\in W^u(x).
\end{aligned}
$$
  \end{proposition}
  
\begin{proof} We will give the proof for  the stable holonomies.
The argument for the unstable holonomies is similar.
Under the fiber bunching condition ``at each step",
 \begin{equation}\label{each step}
\| \B_x\|\cdot \|\B_x^{-1}\| \cdot  \nu(x)^\beta <1 \quad\text{for all }x\in\M,
\end{equation}
existence of such holonomies was proved in \cite{V,ASV} and uniqueness 
in \cite{KS12}.
We indicate how to extend  these results to our setting.

Since the cocycle  $\A$ is fiber bunched
(in the sense of Definition \ref{fiber bunched}) and $\nu<1$, 
there exist $N\in \N$ such that for all  $n\ge N$ and $x\in \M$,
$
\,\| \A_x^n\|\cdot \|(\A_x^n)^{-1}\| \cdot  (\nu^n_x)^\beta < 1.\; 
$
Thus the cocycles $\A^n$ satisfy \eqref{each step}  and hence  have 
unique stable  holonomies.

The stable holonomies for $\A^N$ and $\A^{N+1}$
are also the stable holonomies for  $\A^{N(N+1)}$,
 and hence they coincide by uniqueness. Let $H=H^{\A^{N+1}, s}= H^{\A^N, s}$.
 Clearly, $H$ satisfies the properties (H 1,2,4).
 Also, since $H^{\A^{N+1}, s}$ and $H^{\A^N, s}$ satisfy (H3),
$$
 H_{x,y}=(\A^N_y)^{-1}\circ H_{f^Nx ,\,f^Ny} \circ \A^N_x
 = (\A^{N+1}_y)^{-1}\circ H_{f^{N+1}x ,\,f^{N+1}y} \circ \A^{N+1}_x.
 $$ 
 Hence
$$
H_{f^Nx ,\,f^Ny} = (\A_{f^N y})^{-1}\circ H_{f^{N+1}x ,\,f^{N+1}y} \circ \A_{f^N x},
$$
 and it follows that  $H$ satisfies (H3). The stable holonomy for $\A$ 
 satisfying (H4) is unique 
 since it is also a holonomy for $\A^N$. 
Thus $H=H^{\A,s}$, and  it remains to show that it equals the limit.
 
 By (H3),  
 $\,\A^n_x= (H_{f^n x, f^n y})^{-1 }\circ \A^n_y \circ H_{x,y}^{\A,s},\;$
and hence by (H4) there is a constant $c_1$ 
 such that 
\begin{equation} \label{Ax Ay}
  \|\A^n_x \|= c_1 \,\|\A^n_y\| \quad \text{for all }\;x\in \M, \;\;
  y\in W^s_{loc}(x), \text{ and }n\in \N.
\end{equation}
Hence 
$$
\| H^{\A,s}_{x,y} - (\A^n_y)^{-1} \circ \A^n_x \| \,=\,
\| (\A^n_y)^{-1}\circ (H^{\A,s}_{f^nx ,\,f^ny} -\Id ) \circ \A^n_x \| \,\le
$$
$$
\le \| (\A^n_y)^{-1}\| \cdot \| \A^n_y \| \cdot c\,\dist (f^nx ,\,f^ny)^{\beta} \,\le\,
c_2 \| (\A^n_y)^{-1}\| \cdot \| \A^n_y \| \cdot (\nu^n_y)^{\beta} \to 0
$$
as $n\to \infty\,$ by \eqref{dist} and fiber bunching.  
\end{proof}

\subsection{Relations between H\"older conjugacies and holonomies.}


\begin{proposition} \label{prop intertwines}
Let $\A$ and $\B$ be two fiber bunched cocycles 
and let  $C$ be a $\beta$-H\"older continuous  conjugacy between 
$\A$ and $\B$. Then 
\begin{itemize}
\item[(a)] $C$  intertwines the holonomies 
for $\A$ and $\B$, i.e. 
$$
H_{x,y}^{\A,\,s (u)}=C(y)\circ H_{x,y}^{\B,\,s (u)} \circ C(x)^{-1}
\quad\text{for every }x\in \M \text{ and }y\in W^{s (u)}(x).
$$

\item[(b)] $C$ conjugates the\, {\em periodic cycle functionals} of $\A$ and $\B$, i.e.
$$
H^{\A,s}_{y,x} \circ H^{\A,u}_{x,y}=C(x)\circ H^{\B,s}_{y,x}\circ H^{\B,u}_{x,y} \circ C(x)^{-1}\; 
$$
for every $x\in \M$ and  $y\in  W^s(x)\cap W^u(x)$.
\vskip.15cm

 \item[(c)] $C$ is uniquely determined by its value at one point.

\end{itemize}

\end{proposition}

\begin{proof}
(a) Let $x\in\M$ and $y\in W^s(x)$. By iterating $x$ and $y$ forward
the problem  reduces  to the case of $y\in W^s_{loc}(x)$.
Since $\A(x)=C(fx) \circ \B_x \circ C(x)^{-1}$,
we have
\begin{equation}\label{AA}
\begin{aligned}
&(\A^n_y)^{-1} \circ  \A^n_x 
\,=\, C(y) \circ (\B^n_y)^{-1} \circ C(f^n y)^{-1} \circ  C(f^nx) \circ \B^n_x
\circ  C(x)^{-1}= \\
&= C(y)\circ (\B^n_y)^{-1} \circ (\Id +r_n) \circ \B^n_x\circ C(x)^{-1}= \\
&= C(y) \circ (\B^n_y)^{-1} \circ \B^n_x\circ C(x)^{-1} + 
C(y)\circ (\B^n_y)^{-1} \circ r_n \circ \B^n_x \circ C(x)^{-1}. 
\end{aligned}
\end{equation}
H\"older continuity of $C$  and \eqref{dist} imply that
$$
\begin{aligned}
\| r_n\| & = \|C(f^n y)^{-1} \circ  C(f^nx)-\Id\|
\le \| C(f^n y)^{-1} \|\cdot \| C(f^nx)- C(f^ny)\| \le\\
& \le c_2 \,\dist (f^nx,\, f^ny)^\beta \le c_2 \,(\nu^n_y)^\beta.
\end{aligned}
$$
Using \eqref{Ax Ay},  the above estimate, 
and fiber bunching of the cocycle $\B$, we obtain
$$ 
\|(\B^n_y)^{-1} \circ r_n \circ \B^n_x \| \le 
\| (\B^n_y)^{-1} \| \cdot \| r_n\| \cdot c_3\,\|\B^n_y \|\le
$$ 
$$
 \le c_4\, \| ( \B^n_y)^{-1}\| \cdot \|  \B^n_y \| \cdot  (\nu^n_y)^\beta  \le c_5 \theta^n \to 0 
\quad\text{as }n\to \infty.
$$
Hence the second term in the last line of \eqref{AA} tends to 0.
Since $\underset{n\to\infty}{\lim} (\A^n_y)^{-1} \circ \A^n_x= H^{\A,s}_{x,y}\,$
and $\underset{n\to\infty}{\lim} (\B^n_y)^{-1} \circ \B^n_x= H^{\B,s}_{x,y},\,$
passing to the limit in \eqref{AA} we obtain (a).
 \vskip.1cm
 
The statement for the unstable holonomies is proven similarly and (b) follows immediately from (a).

\vskip.2cm
(c) Let $C(x_0)$ be given. By (a) for every $y\in W^s(x_0)$, the conjugacy at $y$ is given by
$$
  C(y)=H_{y,x}^{\B,s} \circ C(x_0) \circ H_{x,y}^{\A,s}. 
$$
Since the stable manifold $W^s(x_0)$ is dense in $\M$ and $C$ is H\"older continuous,
$C$ is uniquely determined on $\M$.
\end{proof}

\vskip.3cm


\subsection{Cocycles over a diffeomorphism with a fixed point.} $\;$
\vskip.1cm

\noindent {\bf Outline of the proof of Theorem \ref{conjugate  data}.}
Since the cocycle $\A$ is fiber bunched and $\B$ has 
conjugate periodic data, $\B$ is also fiber bunched by Proposition
\ref{fiber bunching from periodic}.
The  theorem then follows from Propositions  
\ref{equal functionals} and \ref{extends} below. 
Somewhat more directly, the argument can be outlined as follows.
We consider the cocycle $\tilde \B =C(p)\circ \B \circ C(p)^{-1}$, 
so that $\tilde \B_p =\A_p,$ and 
the function $\tilde C(q)=C(q) C(p)^{-1}$, so that $\tilde C(p)=\Id$.  
We construct conjugacies between $\A$ and $\tilde \B$
 along the stable and unstable manifolds of $p$ 
$$
\begin{aligned}
 &\tilde C^s(x)=H^{\A,s}_{p,\,x}\circ H^{\tilde \B,s}_{x,\,p} \quad\text{for }x\in W^s(p),\\
 &\tilde C^u(x)=H^{\A,u}_{p,\,x}\circ H^{\tilde \B,u}_{x,\,p} \quad\text{for }x\in W^u(p).
 \end{aligned}
$$
The proof of  Proposition \ref{equal functionals} shows
that if $x$ is a homoclinic point for $p$, i.e. $x\in W^s(p)\cap W^u(p)$,
then 
$$
H^{\A,s}_{x,p} \circ H^{\A,u}_{p,x}= H^{\tilde \B,s}_{x,p}\circ H^{\tilde \B,u}_{p,x},
\quad \text{i.e. }\; \tilde C^s(x)=\tilde C^u(x)\overset{def}=\tilde C(x).
$$
The proof of Proposition \ref{extends} shows that $\tilde C$ is $\beta$-H\"older continuous 
on the set
of homoclinic points, and hence it can be extended to $\M$.
$C(x)= \tilde C(x) C(p)$ is a conjugacy between $\A$ and $\B$,
and it is clear from the construction that it takes values in the 
closed subgroup $G_0$. 
Uniqueness follows from Proposition \ref{prop intertwines}(c).
$\QED$

\vskip.2cm
\noindent {\bf Assumptions.}  In Propositions \ref{equal functionals} and \ref{extends},
 the diffeomorphism $f$ has a fixed point $p$ and 
the cocycles $\A$ and $\B$ are fiber bunched.

\begin{proposition}  \label{equal functionals}
Suppose that  for each periodic point $q=f^k q$ in  a neighborhood 
$U$ of $p$ 
there is $C(q)\in GL(d,\R)$ such that
$$
\A^{k}_q=C(q)\circ \B^{k}_q \circ C(q)^{-1} \quad\text{and} \quad d(C(p), C(q))  \le c\, \dist (p,q)^\beta .
$$
Then $C(p)$ conjugates the  periodic cycle functionals of $\A$ and $\B$ at $p$, i.e.
 $$
  H^{\A,s}_{x,p} \circ H^{\A,u}_{p,x}=
  C(p)\circ H^{\B,s}_{x,p} \circ H^{\B,u}_{p,x}\circ C(p)^{-1}
  \quad\text {for every }x\in W^s(p)\cap W^u(p).
 $$
\end{proposition}

The next proposition describes a sufficient condition
for a conjugacy at a fixed point to  extend to a conjugacy between cocycles. 

\begin{proposition}\label{extends}
 Let  $\,C_p\in GL(d,\R)$ be such that 
\begin{itemize}
\item[(a)] $\;\A_p=C_p \circ \B_p \circ C_p^{-1} $ and
\vskip.1cm
\item[(b)]  $\;H^{\A,s}_{x,p} \circ H^{\A,u}_{p,x}=C_p\circ H^{\B,s}_{x,p}\circ H^{\B,u}_{p,x} \circ C_p^{-1}\;$ for every  $x\in W^s(p)\cap W^u(p)$.
\end{itemize}
Then there exists a unique $\beta$-H\"older continuous conjugacy $C(x)$ between
$\A$ and $\B$ such that $C(p)=C_p$. Moreover, if $\A$ and $\B$
take values in a closed subgroup $G_0$ of $GL(d,\R)$ and $C_p\in G_0$,
then $C(x)\in G_0$ for all $x$.
\end{proposition}

We note that the first assumption on $C_p$ is obviously necessary,
and so is the second one by Proposition \ref{prop intertwines} (b).
Thus a conjugacy $C_p$ between the matrices $\A_p$ and $\B_p$ extends
to a conjugacy between cocycles if and only if (b) is satisfied.

\vskip.3cm

\noindent {\bf Proof of Proposition \ref{equal functionals}.}
First we modify the cocycle $\B$ so that the two cocycles coincide 
at the fixed point $p$.
We define the cocycle $\tilde \B$ and the function $\tilde C(q)$  
 by
$$
\tilde \B_x= C(p)\circ \B_x \circ C(p)^{-1} \quad\text{and}\quad
\tilde C(q)= C(q) C(p)^{-1}, \quad q\in U.
$$
The cocycle $\tilde \B$ is fiber bunched  and $\tilde \B_p =\A_p.$ Also,  
$ \A^k_q=\tilde C(q)\circ \tilde \B^k_q \circ \tilde C(q)^{-1}$ and 
$$
 d(\tC(q), \Id)  \le \tilde c\, \dist (p,q)^\beta 
 \; \text{ for all }q\in U.
$$
We prove that for every 
$x\in W^s(p)\cap W^u(p),\,$ 
 $$
H^{\tilde \B,u}_{p,\,x} \circ H^{\tilde \A,u}_{x,\,p} \circ H^{\tilde \A,s}_{p,\,x} \circ H^{\tilde B,s}_{x,\,p} =\Id.
$$
By Proposition \ref{existence of holonomies}, 
$\,H^{\tilde \B} =
C(p)\circ H^{\B} \circ C(p)^{-1}$, and Proposition \ref{equal functionals} follows.

\vskip.2cm

In the rest of the proof, we write $\B$ for $\tB$
and $C$ for $\tC$ to simplify the notations, and   
we fix $x\in W^s(p)\cap W^u(p).\,$  By Proposition \ref{existence of holonomies},
$$
H^{\A,s}_{p,\,x} \circ H^{\B,s}_{x,\,p} =
\underset{n\to\infty}{\lim}
 \left(  (\A^n_x)^{-1} \circ \A^n_p \circ (\B^n_p)^{-1} \circ  \B^n_x \right)=
 \underset{n\to\infty}{\lim}
 \left(  (\A^n_x)^{-1}\circ  \B^n_x \right)
$$
since  $\B^n_p = \A^n_p$.\, Similarly, 
$$
H^{\B,u}_{p,\,x} \circ H^{\A,u}_{x,\,p} =
 \underset{n\to\infty}{\lim}
 \left(   \B^n_{f^{-n}x} \circ  (\A_{f^{-n}x}^n)^{-1} \right).
$$
Thus,
$$
H^{\B,u}_{p,\,x} \circ H^{\A,u}_{x,\,p} \circ H^{\A,s}_{p,\,x} \circ H^{\B,s}_{x,\,p}
 =  \underset{n\to\infty}{\lim}
 \left(  \B^n_{f^{-n}x} \circ  (\A_{f^{-n}x}^n)^{-1} \circ  (\A^n_x)^{-1}\circ  \B^n_x  \right),
 $$
 and we will show that the limit on the right hand side equals the identity.
\vskip.1cm

 Since $x\in W^s(p)\cap W^u(p)$, by \eqref{Anosov def} 
there is a constant $c_1=c_1(x)$ such that 
$$ 
\dist (f^nx,\, p)< \nu^n_x \cdot c_1 \dist_{W^s(p)}(x,p)  \;\text{ and }\;
\dist (f^{-n}x, \,p)< \hat\nu^{-n}_x\cdot c_1 \dist_{W^u(p)}(x,p),
$$
  and hence 
$$\dist (f^nx,\, f^{-n}x)<c_2 \, \max \{\nu^n_x,\hat\nu^{-n}_x \} .
$$
Therefore, for all sufficiently large $n$ we can  apply  Anosov Closing Lemma 
to the orbit segment $\{f^i (x), \; i=-n, \dots , n\}$ \cite[Theorem 6.4.15]{KH}.
Thus there exists a periodic point $q=f^{2n}q\,$ such that 
$$
 \dist (f^i x,\, f^i q)\le c_3
  \max \{\nu^n_x,\hat\nu^{-n}_x \} \quad \text{for }i=-n,\, \dots, n.
$$
Additionally, we assume that $n$ is large enough so that
$f^{-n}q\in U$.
\vskip.1cm

Now we express  $\B^n_{f^{-n}x},$ 
$(\A_{f^{-n}x}^n)^{-1}\circ(\A^n_x)^{-1}$, and  $\B^n_x  $ in terms of the values 
of the cocycles at the corresponding iterates of $q$.
To use the holonomies, we consider the point
$$ 
z=W^s_{loc}(q)\cap W^u_{loc}(x).
$$ 
It is easy to see that for $i=-n, \dots , n$,
\begin{equation}\label{close}
\dist (f^i z,\, f^i x)\le c_4 \max \{\nu^n_x,\,\hat\nu^{-n}_x \} \;\text{ and }\;
\dist (f^i z,\, f^i q)\le c_4  \max \{\nu^n_x,\,\hat\nu^{-n}_x \}.
\end{equation}
Since $f^i z\in W^u_{loc}(f^i x)$ and $f^i z\in W^s_{loc}(f^i q)$, 
by the properties (H3) and (H3$'$) we have
$$
\B^{n}_x=  H^{\B,u}_{f^{n}z,\,f^{n}x}\circ \B^n_z \circ H^{\B,u}_{x,\,z}=
  H^{\B,u}_{f^{n}z,\,f^{n}x}\circ H^{\B,s}_{f^{n}q,\,f^{n}z}\circ 
\B^n_q \circ H^{\B,s}_{z,\,q} \circ H^{\B,u}_{x,\,z}.
$$
It follows from (H4) that 
$$
H^{s, \B}_{z,q}=\Id+R^{s, \B}_{z,q}, \quad\text{where }\;
\|R^{s, \B}_{z,q}\| \le c \,\dist(z,q)^\beta \le c_5  (\max \{\nu^n_x,\,\hat\nu^{-n}_x \})^\beta,
$$
and similar estimates hold for the other holonomies  due to \eqref{close}.
Thus we obtain 
\begin{equation}\label{r}
\B^{n}_x=(\Id +R_1^n)\circ \B^n_q \circ (\Id +R_2^n), \quad\text{where }
\|R_1^n\|, \|R_2^n\|  \le c_6  (\max \{\nu^n_x,\hat\nu^{-n}_x \})^\beta.
\end{equation}
Similarly,
\begin{equation} \label{1}
\B^n_{f^{-n}x}=(\Id +R_3^n)\circ\B^n_{f^{-n}q}\circ (\Id +R_4^n)
\end{equation}
\begin{equation}\label{2}
(\A_{f^{-n}x}^n)^{-1}\circ(\A^n_x)^{-1}= (\A^{2n}_{f^{-n}x})^{-1}=
(\Id +R_5^n)\circ(\A^{2n}_{f^{-n}q})^{-1}\circ (\Id +R_6^n).
\end{equation}
\vskip.1cm

Since $f^{-n}q\,$ is a point of period $2n$ in the neighborhood 
$U$ of $p$, by the assumption 
there exists $C(f^{-n}q)$ such that 
\begin{equation} \label{3}
\begin{aligned}
 &\A^{2n}_{f^{-n}q}=C(f^{-n}q) \circ \B^{2n}_{f^{-n}q} \circ C(f^{-n}q)^{-1},
 \quad\text{where } \\
 &C(f^{-n}q)=\Id+R_7^n  \quad\text{ and } C(f^{-n}q)^{-1}=\Id+R_8^n \quad\text{ with }\\
 & \|R_7^n\|, \; \|R_8^n\|, \le c_7\, \dist (p,f^{-n}q)^\beta  \le c_8 (\max \{\nu^n_x,\hat\nu^{-n}_x \})^\beta.
 \end{aligned}
\end{equation}
Using  \eqref{2} and \eqref{3} and combining terms of type $\Id+R^n_i$,
we obtain
\begin{equation} \label{4}
(\A_{f^{-n}x}^n)^{-1}\circ(\A^n_x)^{-1}= (\Id +R_9^n)\circ(\B^{2n}_{f^{-n}q})^{-1}\circ (\Id +R_{10}^n)
\end{equation}

Finally \eqref{r}, \eqref{1}, and \eqref{4} yield 
\begin{equation}\label{estimate}
  \B^n_{f^{-n}x} \circ  (\A_{f^{-n}x}^n)^{-1} \circ  (\A^n_x)^{-1}\circ  \B^n_x   =
\end{equation}
 $$
= (\Id + R_3^n) \circ \B^n_{f^{-n}q} \circ  (\Id + R_{11}^n) \circ 
(\B_{f^{-n}q}^n)^{-1} \circ  (\B^n_q)^{-1} \circ
 (\Id +R_{12}^n)\circ \B^n_q \circ (\Id + R_2^n)
$$
$$
=\Id+ \B^n_{f^{-n}q} \circ   R_{11}^n \circ (\B_{f^{-n}q}^n)^{-1} +
(\B^n_q)^{-1} \circ  R_{12}^n\circ \B^n_q \,+
$$
 $$
 +\,\B^n_{f^{-n}q} \circ  R_{11}^n \circ (\B_{f^{-n}q}^n)^{-1} \circ  
 (\B^n_q)^{-1} \circ
R_{12}^n\circ \B^n_q\,+\text{ smaller terms},
$$
where 
$\;\| R^n_i \|\le  c_9 (\max \{\nu^n_x,\hat\nu^{-n}_x \})^\beta.\;$
Since the cocycle $\B$ is fiber bunched,
$$
 \|(\B^n_q)^{-1}\| \cdot  \| \B^n_q\| \cdot \|  R_i^n\| \le c_{10} \, \theta^n 
 \quad\text{and}\quad
 \|\B^n_{f^{-n}q}\| \cdot  \| (\B_{f^{-n}q}^n)^{-1}\| \cdot \|  R_i^n\| \le c_{11} \, \theta^n.
 $$
 Thus we conclude that 
 $$
   \B^n_{f^{-n}x} \circ  (\A_{f^{-n}x}^n)^{-1} \circ  (\A^n_x)^{-1}\circ  \B^n_x   =
   \Id +R^n, \quad \text{where }\; \|R^n\|\le c_{12}\,\theta^n \to 0 \text{ as }n\to \infty,
$$
and hence 
 $$
H^{\B,u}_{p,\,x} \circ H^{\A,u}_{x,\,p} \circ H^{\A,s}_{p,\,x} \circ H^{\B,s}_{x,\,p} =\Id.
$$
This completes the proof of Proposition \ref{equal functionals}. $\QED$

\vskip.4cm 

\noindent {\bf Proof of Proposition \ref{extends}.}
We define a conjugacy $C^s$ on the stable manifold of $p$,
\begin{equation}\label{C^s}
C^s(x)= H^{\A,s}_{p,x}\circ C_p\circ H^{\B,s}_{x,p} \quad\text{for }\,x\in  W^s(p).
\end{equation}
Clearly, $C(p)=C_p$.  Also,
$$
\begin{aligned}
\A^n_x &= H^{\A,s}_{p, \,f^n x}\circ \A^n_p \circ H^{\A,s}_{x,\,p}=
H^{\A,s}_{p, \,f^n x}\circ C_p\circ \B^n_p \circ C_p^{-1}\circ H^{\A,s}_{x,\,p}=\\
&= \,H^{\A,s}_{p, \,f^nx } \circ C_p \circ H^{\B,s}_{f^nx, \,p} 
\circ \B^n_x \circ H^{\B,s}_{p,\,x} \circ C_p^{-1}\circ H^{\A,s}_{x,\,p} \,=\,
C^s(f^nx) \circ \B^n_x \circ C^s(x)^{-1}.
\end{aligned}
$$
Similarly, we define a conjugacy $C^u$ along the unstable manifold of $p$,
$$
C^u(x)= H^{\A,u}_{p,x}\circ C_p\circ H^{\B,u}_{x,p} \quad\text{for }\,x\in  W^u(p).
$$
Let $X=W^u(p) \cap W^s(p)$ be the set of homoclinic points of $p$.
By the assumption (b), 
\begin{equation}\label{=}
C^s(x)=C^u(x)\overset{def}{=}C(x) \quad\text{ for every}\, x\in X.
\end{equation}
\vskip.2cm

The set of homoclinic points of $p$ is known to be dense in $\M$
\cite{Bo}. 
To extend the function $C$ from $X$ to $\M$, we show that  $C$ 
is H\"older continuous on  $X$. Let $x$ and $y$ be two sufficiently close 
points in $X$. We note that the distances between $x$ and $y$ 
along $W^s(p)$ and $W^u(p)$ can be large.
To make an estimate we consider the point 
$$
z=W^u_{loc}(x)\cap W^s_{loc}(y),
$$ 
which is also in $X$.
By the definition of $C=C^s$ and properties of holonomies,
$$
C(z)=  H^{\A,s}_{y,z}\circ C(y)\circ H^{\B,s}_{z,y} =
(\Id +R^{\A}_{y,z}) \circ C(y)\circ (\Id +R^{\B}_{y,z}),
$$
where $\|R^{\A}_{y,z}\|, \; \|R^{\B}_{y,z}\| \le c\,\dist (y,z)^\beta.\;$ 
Hence 
\begin{equation}\label{CC^s}
\begin{aligned}
& C(z) \circ C(y)^{-1} = (\Id +R^{\A,s}_{y,z} )\circ  C(y) \circ (\Id +R^{\B,s}_{y,z}) \circ C(y)^{-1}= \\
& =  \; \Id+R^{\A,s}_{y,z}+ C(y) \circ R^{\B,s}_{y,z} \circ C(y)^{-1} +
R^{\A,s}_{y,z} \circ  C(y) \circ R^{\B,s}_{y,z} \circ C(y)^{-1}.
\end{aligned}
\end{equation}

\noindent Similarly, using unstable holonomies, we obtain
\begin{equation}\label{CC^u}
\begin{aligned}
&C(x)\circ C(z)^{-1}= \\
&=\;\Id+R^{\A,u}_{x,z}+ C(z) \circ R^{\B,u}_{x,z} \circ C(z)^{-1} +
R^{\A,u}_{x,z} \circ  C(z) \circ R^{\B,u}_{x,z} \circ C(z)^{-1}.
\end{aligned}
\end{equation}

Now we show that $\|C\|$ and $\|C^{-1}\|$ are bounded on $X$.
We fix a small number $\e$
and choose a finite subset $Y$ of $X$ such that for each
$x\in X$  there is $y \in Y$ such that $\dist (x,y)\le \e$.
Since $Y$ is finite, there is a constant $M$ such that 
$$
  \|C(y)\| \le M \quad\text{and}\quad \|C(y)^{-1} \|\le M
  \quad\text{for all }y\in Y.
$$

Let $x\in X$, let $y \in Y$ be such that $\dist (x,y)\le \e$, and let 
$z=W^s_{loc}(x)\cap W^u_{loc}(y)$. Then multiplying both sides 
of  \eqref{CC^s} by $C(y)$ and estimating  the norm we see
that 
$$ \|C(z)\|\le (2+2M^2)M,
$$ 
assuming that $\e$ is sufficiently small so that $c\,\dist(x,z)^\beta<1$.
Now boundedness of $\|C(x) \|$ follows similarly from  \eqref{CC^u}.
One can obtain expressions for $C(z)^{-1}\circ C(y)$ and 
$C(x)^{-1}\circ C(z)$ similar to \eqref{CC^s} and \eqref{CC^u}
and  conclude that $\|C(x)^{-1} \|$ is also bounded on $X$.

Now it follows from \eqref{CC^s} and \eqref{CC^u} 
that for any sufficiently close $x,y$ in $X$
$$
\begin{aligned}
C(x)\circ C(y)^{-1}&=C(x)\circ C(z)^{-1}\circ C(z)\circ C(y)^{-1} =
\Id +R_{x,y},\\
&\text{where}\quad  \|R_{x,y}\| \le c'\dist (x,y)^\beta,
\end{aligned}
$$
and hence 
\begin{equation} \label{d(CC)}
\begin{aligned}
 & d(C(x),C(y))=\|C(x)-C(y)\|+ \|C(x)^{-1}-C(y)^{-1}\|\le\\
 &\le  \|C(x)C(y)^{-1}-\Id \|\cdot \|C(y)\|+
 \|C(x)^{-1}\|\cdot \|\Id-C(x)C(y)^{-1}\| \le \\ &\le 2c'  M'\,\dist(x,y)^\beta.
\end{aligned}
\end{equation}
Thus we can extend the function $C$ on $X$ to 
a $\beta$-H\"older continuous function on $\M$,
and 
$$
\A^n_x=C(f^nx) \circ \B^n_x \circ C(x)^{-1} \quad \text{for all }
x\in \M \text{ and }n\in \Z.
$$

\newpage

The conjugacy $C$ takes values in the closed subgroup $G_0$ by the construction:
the holonomies take values in $G_0$ by Proposition \ref{existence of holonomies},
hence so does the restriction of $C$ to $X$ by \eqref{C^s} and \eqref{=},
and thus so does  $C$.
Uniqueness of the conjugacy follows from Proposition \ref{prop intertwines}(c).
$\QED$


\vskip.5cm


\subsection {Centralizers of cocycles and connections to conjugacies.} $\;$
\vskip.1cm
 
\noindent The {\em centralizer}\, of a cocycle  of $\A$ is the set
 $$
    Z(\A)=\{ D:\M\to GL(d,\R)\; |\;\; \A_x=D(fx)\circ \A_x \circ D (x)^{-1} 
    \quad\text{for all }x\in \M \}.
 $$
We consider the  centralizer  in the $\beta$-H\"older category.
 
 It is easy to see that $Z(\A)$ is a group with respect to pointwise multiplication and that $Z(\A)$ is a subgroup of $Z(\A^{k})$ for all $\,k\ge 1$. 
\vskip.1cm

\begin{proposition} \label{centralizers stabilize}
For any fiber bunched cocycle $\A$ there exists  
$M\ge 1$ such that 
$$\;Z(\A^{MT}) = Z(\A^M)\;\text{ for all  }\;T\ge 1.
$$
\end{proposition}

\begin{proof}  We note that for every $k\ge 1$ the cocycle $\A^k$ are also fiber bunched. 

Let $p$ be a periodic point of $f$ of period $N$. Then it is a fixed point
for $f^N$, and we consider the iterate $\bar \A =\A^N$ 
over $f^N$.
An element $D$ of the centralizer of $\bar \A^k$
is a conjugacy between $\bar \A^k$ and itself.
Hence by Proposition~\ref{prop intertwines}, $D$ is uniquely determined by 
its value  at $p$.
By Proposition \ref{extends}, a matrix $D_p=D(p)$ extends to 
a H\"older conjugacy $D$ on $\M$ if and only if 
\vskip.25cm

 $\hskip1cm \bar \A_p^{k}=D_p \circ \bar \A_p^{k} \circ D_p^{-1}\; \;$ and
\vskip.1cm
 $\hskip1cm  H^{\bar \A^k,s}_{x,p} \circ H^{\bar \A^k,u}_{p,x}
 =D_p\circ H^{\bar \A^k,s}_{x,p}\circ H^{\bar \A^k,u}_{p,x} \circ D_p^{-1}\;$ for every  $x\in W^s(p)\cap W^u(p)$.

\vskip.25cm

\noindent The second condition is the same for all $k\ge 1$ since
the holonomies of $\bar \A$ coincide with the holonomies of $\bar \A^k$
by the uniqueness.

The first condition is equivalent to the system of linear equations 
$\bar \A_p^k \circ D_p=D_p \circ \bar \A_p^k$ in $d^2$ variables,
and hence the set of its solutions can be identified with a subspace $V_k$
of $\R^{d^2}$. Intersecting this set with $GL(d,\R)$ gives the
centralizer of the matrix $\bar \A^k_p$. 
The dimensions of the subspaces $V_k$ are bounded by $d^2$.
Let $L\ge 1$ be the smallest  number such that $\dim V_L=\max_k \dim V_k$.
Clearly $V_{L}\subseteq V_{LT}$, and hence $V_{L}= V_{LT}$.
Therefore,
$$
Z(\bar \A^{LT}) = Z(\bar \A^L) \quad\text{i.e.}\quad
Z(\A^{NL\cdot T}) = Z(\A^{NL})   \quad\text{for all }\;T\ge 1.
$$
\end{proof}

The following proposition is easy to verify.

\begin{proposition} \label{conj and Z}
Let $\,\mathcal{C}(\A, \B)$ be the set of conjugacies 
between $\A$ and $\B$, 
and let $\,C_1 \in \mathcal{C}(\A, \B)$.
Then $C_2 \in \mathcal{C}(\A, \B)$ \ if and only if $\;C_1C_2^{-1}\in Z(\A)$. Thus the conjugacy between $\A$ and $\B$ is unique up to 
an element of the centralizer and 
$\,\mathcal{C}(\A, \B)=Z(\A) C_1.$

\end{proposition}

\newpage

\subsection{Proof of Theorem \ref{equal data}} 

To obtain a fixed point,  we pass to an iterate of $f$.
Let $p_1$ be a periodic point of $f$ of period $N$.
We consider the diffeomorphism $f^N$
and the cocycles $\A^N$ and $\B^N$ over $f^N$.
Clearly these cocycles are $\beta$-H\"older continuous and 
fiber bunched.
Thus we apply Theorem \ref{conjugate data} with $C(p)=\Id$
and conclude that  there exists a $\beta$-H\"older continuos 
conjugacy $C_1$ between
 $\A^N$ and $\B^N$. 
It remains  show that there exists a conjugacy beween the original
 cocycles $\A$ and $ \B$ over $f$.
 
By Proposition \ref{centralizers stabilize} there exists $M$ such that
$Z(\A^{NM\cdot T}) = Z(\A^{NM})$ for every $T\ge 1$. 
We note that $C_1$ is also a conjugacy for $\A^{NM}$ and $\B^{NM}$.

It is known that any transitive Anosov diffeomorphism has periodic points 
of all sufficiently large periods.
We pick a periodic point $p_2$ of a period $K>1$ relatively
 prime with $MN$. As above, we obtain 
 a conjugacy $C_2$ for the cocycles $\A^K$ 
 and $\B^K$ over $f^K$. 
 Thus both $C_1$ and $C_2$ are H\"older conjugacies 
 for the cocycles $\A^{NMK}$ and $\B^{NMK}$ over $f^{NMK}$, and hence 
by Proposition \ref{conj and Z},  $C_1C_2^{-1} \in Z(\A^{NMK})= Z(\A^{NM})$. 
 Since $C_1$ is a conjugacy for $\A^{NM}$ and $\B^{NM}$,
$\;C_2$ is also a conjugacy for these cocycles.

Thus $C_2$ is a conjugacy for the cocycles over $f^{NM}$ and $f^K$,
where  $MN$ and $K$ are relatively prime.
Hence there exist integers $r$ and $s$ such that
$NMr+Ks=1$, and
it is easy to see  that $C_2$ is also a conjugacy 
for the cocycles $\A$ and $\B$ over $f$.

This completes the proof of the theorem. $\QED$
\vskip.3cm


\subsection{Proof of Theorem \ref{measurable conjugacy} }
Since the cocycle $\B$ is uniformly quasiconformal
(see Definition~ \ref{quasiconformal}),
it satisfies the fiber bunching condition \eqref{fiber bunched} with
$$
  L= \sup_{x, n} \, (\|\B^n_x\|\cdot \| (\B^n_x)^{-1}\| )\quad\text{and}\quad
  \theta=\max_x \,\nu(x).
$$

Let $C$ be a $\mu$-measurable conjugacy between $\A$ and $\B$.
First we show that $C$ intertwines holonomies of $\A$ and $\B$
on a set of full measure, i.e.
there exists a set $Y\subset \M$, $\,\mu(Y)=1$,  such that 
\begin{equation}\label{int meas}
H_{x,y}^{\A,s}=C(y)\circ H_{x,y}^{\B,s} \circ C(x)^{-1}
\quad\text{for all }x,y\in Y\; \text{ such that }y\in W^s(x),
\end{equation}
and a similar statement holds for the unstable holonomies.
\vskip.1cm

Let $x\in \M$ and $y\in W^s(x)$.
As in the proof of Proposition \ref{prop intertwines}(a), we obtain that 
\begin{equation}\label{AA2}
(\A^n_y)^{-1} \circ  \A^n_x =  C(y) \circ (\B^n_y)^{-1} \circ \B^n_x\circ C(x)^{-1} + 
C(y)\circ (\B^n_y)^{-1} \circ r_n \circ \B^n_x \circ C(x)^{-1}, 
\end{equation}
where
$$
\| r_n\|  \le \,\| C(f^n y)^{-1} \|\cdot \| C(f^nx)- C(f^ny)\| .
$$

Since $C$ is $\mu$-measurable, by Lusin's theorem there exists a 
compact set $S\subset \M$ with $\mu(S)>1/2$ such that $C$
is uniformly continuous on $S$ and hence $\|C\|$ and $\|C^{-1}\|$
are bounded on $S$.
Let $Y$ be the set of points in $\M$ for which the frequency of 
visiting $S$ equals $\mu(S)>1/2$. By Birkhoff Ergodic Theorem
$\mu(Y)=1$. If $x$ and $y$ are in $Y$, there exists a sequence $\{n_i\}$
such that $f^{n_i}x$ and $f^{n_i}y$ are in $Y$ for all $i$.
It follows that 
$$
  \|r_{n_i}\| \to 0 \quad \text{as }i\to\infty 
$$
and $\|C\|$, $\|C^{-1}\|$ are uniformly bounded on $\{x_{n_i}, \, y_{n_i}\}$.
The product 
$$ 
\|(\B^n_y)^{-1}\| \cdot \|\B^n_x \| \le
\| H^{\B,s}_{f^nx,\,f^ny}\|\cdot \|( \B^n_x)^{-1}\|\cdot \|H^{\B,s}_{x,y}\|\cdot \| \B^n_x \|
$$ 
is uniformly bounded since the cocycle $\B$ is uniformly quasiconformal.

Thus for every $x$ and $y$ in $Y$ such that $y\in W^s(x)$, the second 
term in \eqref{AA2} tends to 0 along a subsequence, and \eqref{int meas} follows.
The statement for the unstable holonomies is proven similarly.

Let $x,y\in Y$ and $y\in W^s_{loc} (x)$. Then by \eqref{int meas} 
$$
C(y) = H_{x,y}^{\A,s} \circ C(x)\circ H_{x,y}^{\B,s}.
$$
It follows as in the proof of Proposition \ref{extends}, \eqref{CC^s}, that 
$$
C(y) \circ C(x)^{-1} = \; \Id+R^{\A,s}_{x,y}+ C(x) \circ R^{\B,s}_{y,z} \circ C(x)^{-1} +
R^{\A,s}_{y,z} \circ  C(x) \circ R^{\B,s}_{x,y} \circ C(x)^{-1},
$$
where
$$
 \|R^{\A}_{x,y}\|, \; \|R^{\B}_{x,y}\| \le c\,\dist (x,z)^\beta.
$$
Since $C$ is bounded on $Y$, this implies that 
$$
 \|C(y) \circ C(x)^{-1} -\Id\|\le c_1\,\dist(x,y)^\beta,
$$
and it follows as in \eqref{d(CC)} that 
\begin{equation}\label{d}
d(C(x), C(y))\le c_2 \,\dist(x,y)^\beta,
\end{equation}
where $c_2$ does not depend on $x$ and $y$.
The same holds for any $x,y\in Y$ such that  $y\in W^u_{loc}(x)$. 
\vskip.2cm

We consider a small open set $U$ in $\M$ with a product structure,
i.e. 
$$
U=W^s_{loc}( x_0)\times W^u_{loc}( x_0)\overset{def}{=}\,
\{W^s_{loc}( x)\cap W^u_{loc}( y)\; | \; x\in W^s_{loc}( x_0), \;y\in W^u_{loc}( x_0)\}.
$$
Since the measure $\mu$ has local product structure,
$\mu$ is equivalent to the product of conditional measures on $W^s_{loc}( x_0)$
and $W^u_{loc}( x_0)$, and hence for $\mu$ almost all local 
stable leaves in $U$, the set of  points of 
$Y$ on the leaf has full conditional measure.  
Since $\mu$ has full support, the conditional measures on almost all leaves 
have full support.

Hence for any two  points $x$ and $z$ in $Y\cap U$ that lie on two such 
stable leaves, there exists a point 
$y\in W^s_{loc}(x)\cap Y$  such that 
$W^u_{loc}(y)\cap W^s_{loc}(z)$ is also in $Y\cap U$.
It follows from \eqref{d} and the local product structure of the stable 
and unstable manifolds that
$$
d(C(x), C(z))\le c_3 \,\dist(x,z)^\beta.
$$
This  estimate holds for all $x,z$ in a set of full measure
$\tilde Y \subset Y$.

Let $\bar Y = \bigcap_{n=-\infty}^{\infty} f^n(\tilde Y)$. 
Then $\bar Y$ is $f$-invariant and 
$A(x)=C(fx)\circ \B(x) \circ C(x)^{-1}$ for all $x\in \bar Y$. Since $\mu$ has
full support and $\mu(\bar Y)=1$, the set $\tilde Y$ is dense in $\M$. 
Hence we can extend $C$ 
from $\bar Y$ and obtain a H\"older continuous conjugacy
 $\tilde C$ on $\M$ that coincides with $C$ on a set of full measure. 

\vskip1cm


\section{An application: smooth conjugacy  to a small perturbation
for  Anosov automorphisms}  \label{localrig}

Let  $g$ be an  Anosov diffeomorphism of $\M$. If $f$ is
a diffeomorphism of $\M$ sufficiently $C^1$ close to $g$, then $f$ is also Anosov and it is
topologically conjugate to $g$, i.e. there exists a homeomorphism
$h$ of $\M$  such that
$$
g = h^{-1} \circ  f \circ h.
$$
Moreover, the conjugacy is unique when chosen near identity.
(See, e.g. ~\cite[Corollary 18.2.2]{KH}).
The conjugacy $h$ is only H\"older continuous in general, 
and it is important to find out when the diffeomorphisms $f$ and $g$ are 
{\em smoothly}\, conjugate. If  $h$ is a $C^1$ diffeomorphism,
then the derivatives of the return
maps of $f$ and $g$ at the corresponding periodic points are
conjugate. Indeed, differentiating $\,g^n = h^{-1} \circ  f^n \circ h\;$
at periodic points $p=f^n(p)$ yields
$$
 D_pg^n = (D_ph)^{-1} \circ D_{h(p)} f^n \circ D_ph
 \quad\text{whenever }\;p=f^n(p).
$$

A diffeomorphism $g$ is said to be {\em locally rigid}\, if for any
$C^1$-small perturbation $f$ the conjugacy of the derivatives at 
the periodic points   is sufficient for
$h$ to be $C^1$. 
The problem of local
rigidity has been extensively studied and  Anosov diffeomorphisms with
one-dimensional stable and unstable distributions were shown to be
locally rigid \cite{L0,LM, L1}. In general, this is not the case for
systems with higher-dimensional distributions  \cite{L1,L2}.
Positive results were established for certain classes of diffeomorphisms
that are conformal
on the full stable and unstable distributions, \cite{L2,KS03,L3,KS09}.
In a different direction, local rigidity was proved in \cite{G} for
an irreducible Anosov toral automorphism $L : \T^d \to \T^d$
with real eigenvalues of distinct moduli, as well as for some
nonlinear systems with similar structure. Recently,
this result was extended to a broad class of Anosov automorphisms.

\begin{theorem} \cite{GKS}\label{GKS theorem}
Let $L:\T^d\to\T^d$ be an irreducible Anosov automorphism
such that no three of its eigenvalues have the same modulus.
Let  $f$ be a $C^1$-small perturbation of $L$ such that the
derivative $D_pf^n$ is conjugate to $L^n$ whenever $f^n(p)=p$.
Then $f$ is $C^{1+\text{H\"older}}$ conjugate to $L$.
\end{theorem}

We recall that an automorphism $L$ is called
to be {\em irreducible}\, if it has no rational invariant subspaces,
or equivalently if its characteristic polynomial is irreducible over~$\Q$. 
Examples in \cite{G}  show that irreducibility of $L$ is a  necessary
assumption for local rigidity except when $L$ is conformal on
the stable and unstable distributions.

Theorem \ref{conjugate  data} allows us to obtain 
an alternative sufficient condition for smoothness of the conjugacy to a 
small perturbation. 
Instead of the assumption on the eigenvalues of $L$ we make 
an assumption that  the conjugacy of the periodic data
of the cocycles $L=DL$ and $Df$ is H\"older continuous 
at a single periodic point.

\newpage

\begin{theorem} \label{rigidity theorem}
Let $L:\T^d\to\T^d$ be an irreducible Anosov automorphism
and let  $f$ be a $C^1$-small perturbation of $L$. 
Suppose that   for each periodic point $p=f^n (p)$ 
there is $C(p)$ such that
$D_p f^{n}=C(p)\circ L^n \circ C(p)^{-1}$ and 
$C(p)$ is H\"older continuous at a periodic point~$p_0$.
Then $f$ is $C^{1+\text{H\"older}}$ conjugate to $L$.
\end{theorem}

The proof of this theorem differs   from the proof of 
Theorem~\ref{GKS theorem} only in the way we obtain conformality 
of $Df$ on certain invariant sub-bundles, as explained below. 

We denote by $E^{u,L}$  the unstable distribution of $L$. 
Let $\,1<\rho_1 <\rho_2< \dots <\rho_l\,$ be the distinct moduli
of the unstable eigenvalues of $L$, and let
$$
 E^{u,L}= E_1^L \oplus E_2^L \oplus \dots \oplus E_l^L
$$
 be the corresponding splitting of the unstable distribution.
Since $f$ is $C^1$ close to $L$, $f$ is also
Anosov, and its unstable distribution $E^{u,f}$ splits into 
a direct sum of $l$ invariant H\"older continuous distributions close to the
corresponding distributions for $L$:
 $$
   E^{u,f}= E_1^f \oplus E_2^f \oplus \dots \oplus E_l^f
$$
(see, e.g. ~\cite[Section 3.3]{Pes}).

Let $\A=L|{E_i^L}$ and  $\B=Df|{E_i^f}$. 
 Conformality of the cocycles plays an important role in establishing 
 smoothness of the conjugacy.
Since $L$ is irreducible, all its eigenvalues are simple. Thus 
the restriction of $L$ to $E_i^L$ is diagonalizable over $\C$
and its eigenvalues are of the same modulus.
Hence the cocycle $\A=L|{E_i^L}$ is conformal
in some norm.

In \cite{GKS}, conformality of $\B$
at the periodic points together with the assumption that 
the distributions $E_i^L$ and $E_i^f$ are either one- or two-dimensional
allows us to conclude that, by  \cite[Theorem 1.3]{KS10}, 
the cocycle $\B$ is conformal.  In higher dimensions,
conformality at the periodic points 
does not imply conformality  \cite[Proposition 1.2] {KS10}.
In Theorem \ref{rigidity theorem} we make no assumptions 
on the dimensions of $E_i^L$,  and so we use a different approach
to obtain conformality of $\B$.
\vskip.1cm

Let $\beta>0$ be so that $C(p)$ is $\beta$-H\"older at $p_0$
and all cocycles $Df|{E_i^f}$ are $\beta$-H\"older.
 Since the cocycle $\A=L|{E_i^L}$ is conformal,  it is fiber 
bunched and it follows that  the cocycle  $\B=Df|{E_i^f}$
is also fiber bunched.  
We consider the iterates $\A^N$ and $\B^N$ over $f^N$,
where $N$ is the period of $p_0$.
Theorem \ref{conjugate data} implies
that there exists a H\"older continuos conjugacy $C$ between
 $\A^N$ and $\B^N$.  Since $\A^N$ is conformal,
 this implies that $\B^N$ is uniformly quasiconformal,
 and hence so is $\B$.
By \cite[Corollary 3.2] {KS12}, $\B$ is conformal with respect to a continuous Riemannian metric
on $E_i^f$. 

After conformality of $Df$ on each sub-bundle $E_i^f$ is obtained,
the proof of Theorem \ref{rigidity theorem} proceeds exactly
as the proof of Theorem \ref{GKS theorem}.
We  consider the topological conjugacy $h$ between $L$ and $f$ 
close to the identity.
We use conformality to show that $h$ is $C^{1+ \text{H\"older}}$ 
along the leaves of the linear foliation tangent to $E_i^L$,
and then we establish the smoothness of $h$ on $\M$.
$\QED$
\vskip.2cm

Corollary \ref{constant}, which we prove next, shows  that the cocycles
$L$ and $Df$ are H\"older cohomologous without irreducibility assumption
on $L$. This is not known to imply smoothness of $h$. The arguments 
in Theorems \ref{GKS theorem} and \ref{rigidity theorem} use both
density of the subspaces $E^L_i$ in $\T^d$ and conformality of $L|E^L_i$, which
follow from irreducibility. 

\vskip.3cm

\noindent {\bf Proof of Corollary \ref{constant}.}
The proof is closely related to the above argument.

Let $A(x) =A$ be the generator of $\A$.
Let $\rho_1 < \dots <\rho_l$ be the distinct moduli
of the eigenvalues of $A$ and let 
$\R^d = E_1^A \oplus  \dots \oplus E_l^A$ 
be the corresponding invariant splitting into direct sums of the 
generalized eigenspaces. We denote $A_i=A| E_i^A$. It follows 
that for any $\e >0$ there exists $C_\e$ such that 
$$
C_\e^{-1} (\rho_i-\e)^n \le \| A_i^n u \| \le C_\e (\rho_i+\e)^n
\quad \text{for any unit vector }u\in E_i^A,
$$ 
 and hence the  cocycle $\A_i$ generated by 
$A_i$ is fiber bunched for any $\beta>0$. Moreover, any cocycle $\B$ with generator $B$
sufficiently $C^0$ close to $A$ has the corresponding invariant
splitting $\R^d = E_1^B (x) \oplus \dots \oplus E_l^B(x)$,
which is close to that of $\A$ and is $\beta$-H\"older for some $\beta >0$.
The corresponding restrictions $\B_i$ satisfy similar estimates and
hence are also fiber bunched. Since the conjugacy $C(p)$ maps 
$E_i^A(p)$ to $E_i^B(p)$, the cocycles $\A_i$ and $\B_i$ have conjugate 
periodic data. Hence by Theorem \ref{conjugate data} they are conjugate via a H\"older
continuous function $C_i$ and we obtain a conjugacy between 
$\A$ and $\B$ 
as the direct sum of $C_i$.
 $\QED$

\end{document}